\documentclass[preprint]{article}
\pdfoutput=1

\usepackage{array}
\usepackage{amssymb,amsfonts,amsmath,latexsym,dcolumn}
\usepackage{arydshln}
\usepackage{enumerate}
\usepackage{stmaryrd}
\usepackage{times,mathptm}
\usepackage{graphicx}
\usepackage{pifont}
\usepackage{algorithmic}
\usepackage{algorithm}
\newcommand{\be}{\begin{equation}}
\newcommand{\ee}{\end{equation}}
\newcommand{\ba}{\begin{array}}
\newcommand{\ea}{\end{array}}

\newcommand{\comment}[1]{}

\newcommand{\B}[1]{\mathbf{#1}}

\newenvironment{code1}{%
                           \mathcode`\:="603A  
                           \def\colon{\mathchar"303A}
                           \par
                           \upshape
                           \begin{list} 
                              {} {\leftmargin = 0.0cm}
                           \item[]
                           \begin{tabbing}
                              \hspace*{.3in} \= \hspace*{.3in} \=
                              \hspace*{.3in} \= \hspace*{.3in} \=
                                                                                                                                                                                                     \hspace*{.3in} \= \hspace*{.3in} \= \kill
                          }{\end{tabbing}\end{list}}

\newcounter{master}
\newcounter{enumcnt}

\usepackage{amssymb,amsfonts,amsmath,latexsym,dcolumn,amsthm}
\usepackage{arydshln}
\usepackage{enumerate}
\usepackage{stmaryrd}
\usepackage{times}
\usepackage{graphicx}
\usepackage{pifont}
\usepackage{algorithmic}
\usepackage{algorithm}

\newtheorem{corollary}{Corollary}
\newtheorem{definition}{Definition}
\newtheorem{thm}{Theorem}
\newtheorem{lem}[thm]{Lemma}
\usepackage{a4wide}

\begin{document}
\title{Compression  of  unitary  rank--structured matrices  to CMV-like  shape with an application 
to polynomial rootfinding\thanks{This work was partially supported by GNCS-INDAM,  grant ''Equazioni e funzioni di Matrici''}}

\author{Roberto Bevilacqua,  Gianna M. Del Corso and Luca Gemignani\thanks{Universit\`a di Pisa, Dipartimento di Informatica, Largo Pontecorvo, 3, 56127 Pisa, Italy, email: \{bevilacq, delcorso, l.gemignani\}@di.unipi.it}}

\date{June, 2, 2013}
\maketitle

\begin{abstract}
This paper is concerned with the reduction of a unitary matrix $U$  to CMV-like  shape.
A   Lanczos--type  algorithm is presented which carries out the reduction   by computing
the block tridiagonal form of the Hermitian part of $U$, i.e., of the matrix $U+U^H$.
By elaborating on  the Lanczos approach  we also
propose an alternative algorithm using elementary matrices  which is numerically stable.
If $U$ is rank--structured then the same property holds for its Hermitian part and, therefore,
the block tridiagonalization process can be  performed using the rank--structured matrix
technology with reduced complexity.  Our interest
in the  CMV-like reduction is motivated by the unitary  and almost unitary eigenvalue problem.
In this respect,  finally,  we discuss the application of the CMV-like
reduction for the  design of fast companion eigensolvers based on the customary
 QR iteration.

\end{abstract}

\noindent{\small{\bf Keywords}
CMV matrices,  unitary matrices, rank--structured matrices, block tridiagonal reduction, QR iteration,
block Lanczos algorithm, complexity
{\bf MSC}
65F15
}

\newcommand{\bs}{\mbox{\rm {block--span}}}
\newcommand{\rank}{\mbox{\rm {rank }}}

\section{Introduction}
\setcounter{equation}{0}
In a recent paper \cite{KN}, Killip and Nenciu conduct a systematic study
of the properties of a class of unitary matrices --named  CMV matrices from the
 names of the researchers who introduced  the class in \cite{CMV}--.
The rationale of the paper emphasized in the title is that CMV matrices are the unitary
analogue  of  Hermitian
tridiagonal Jacobi matrices. It is  interesting to consider
 potential analogies from the point of view of
numerical computations since  the Hermitian tridiagonal structure provides a very efficient and compressed
representation of Hermitian matrices especially for eigenvalue computation.
To our knowledge the first fundamental contribution  along this line was provided in
the paper \cite{BE91} where some algorithms are presented for the reduction of a unitary matrix to a
CMV-like form and, moreover, the properties of this form under the QR process are investigated. In \cite{VW13} 
a more general framework is proposed, which includes  the CMV-like form as a particular instance.


Our interest in the CMV reduction of unitary matrices stems from the research of numerically
efficient  eigenvalue algorithms  for unitary and almost unitary rank--structured matrices.
 The structural properties of the classical Hessenberg reduction of a unitary matrix
 were extensively exploited for  fast eigenvalue computation in some pioneering papers by Gragg and coauthors
\cite{Gragg3,Gragg2,Gragg1}. More recently,  it has been realized that these properties are related with the
rank structures  of a Hessenberg unitary matrix and, moreover, these structures are invariant under the QR
iteration  thus yielding  a large variety of possible  fast adaptations  of this iterative method
 for eigenvalue computation of unitary and almost unitary  matrices \cite{BEGG_simax,BEGG_mc,DVB}.
A common issue of all these
approaches  is that rank--structured matrices are generally represented in a data--sparse format  and the
accurate update of this format under the iterative process can be logically complicated and time consuming.
Additionally, these difficulties are magnified  for block matrices where the size of the condensed representation
typically increases.
So the point of
view taken in this paper is that, for numerical and computational efficiency of
eigensolvers,
 it would be better to replace the data-sparse format with a sparser  representation
employing directly the matrix entries.  The CMV reduction naturally lead
to a sparser form.

This paper is aimed to explore  the mechanism underlying the reduction
of a general unitary matrix $U\in \mathbb C^{n\times n}$  into  some CMV-like  form.
It is shown
that
an  origin of such mechanism can be found in the property that under some
mild assumption the matrices
$U_H\colon =U +U^H$ and $U_{AH}\colon =U -U^H$ can be simultaneously reduced into
block tridiagonal form by
a unitary congruence, i.e., $U_H\rightarrow T_H=Q^H U_H Q$ and
$U_{AH}\rightarrow T_{AH}=Q^H U_{AH} Q$.  From this it follows that the same transformation acts on
the matrix $U$ by performing the reduction into block tridiagonal form, that is,  $U\rightarrow
\displaystyle\frac{T_H +T_{AH}}{2}=Q^H U Q$. A careful look at the out-of-diagonal blocks of
$\displaystyle\frac{T_H +T_{AH}}{2}$ reveals that
these are of rank one at most by showing   the desired  CMV-like  shape of  the matrix.
The complexity of the block tridiagonal reduction  is generally $O(n^3)$.
However there are some interesting structures which can be
exploited. In particular, if $U$ is rank-structured then $U+U^H$ is also rank-structured and
Hermitian so that  the fast techniques  developed in  \cite{EGG} might
 be incorporated  in the algorithm  by achieving a
quadratic complexity.

The paper is organized as follows.
In Section 2, we describe and analyze a block Lanczos procedure for
 the reduction of
$U_H$ and simultaneously of $U_{AH}$ into block tridiagonal form.
It is also shown that
this procedure is  at the core of our CMV--ification algorithm for unitary matrices.
In Section 3  by elaborating on the Lanczos approach  we
propose a different CMV reduction algorithm in Householder
style using elementary matrices
and comment on  the cost
analysis of the algorithm in the case  of rank--structured inputs.
In Section 4  we discuss the application of the
CMV-like representation of a unitary matrix   in  the design of fast eigensolvers  for
unitary and almost unitary matrices.
Finally, in Section 5
 the conclusion and  further developments are drawn.

\section{Derivation of the Algorithm}
\setcounter{equation}{0}

Let $U\in \mathbb C^{n\times n}$ denote  a unitary matrix.  In this section we consider the  problem of reducing
 $U$ into
a sparse form referred  to as CMV--like shape.

\begin{definition}\label{cmv-like}
We say that an $n\times n$ unitary matrix $U$ has {\it CMV-like}  shape  or, for short, is a
CMV-like matrix
if it is block tridiagonal, that is,
\[
U=\left[\begin{array}{ccccc} U_{1} & \tilde U_{1} \\\hat U_{1} & \ddots &  \ddots \\ & \ddots & \ddots & \tilde U_{p-1}\\
& & \hat U_{p-1} & U_p\end{array}\right],
\]
where $U_{k}\in \mathbb C^{i_k\times i_k}$, $1\leq k\leq p$, $i_1=\ldots=i_{p-1}=2$ and $i_p\in \{1,2\}$
depending on the parity of $n$, and, moreover, both the superdiagonal and the subdiagonal blocks
are matrices of rank one at most.
\end{definition}

In the generic case a CMV--like matrix $U$ can be converted into a matrix with the CMV--shape described  in
Definition 1.2 in \cite{KN} by means of a unitary diagonal congruence. However, our class is a bit more general
since it includes  the direct sum of CMV shaped matrices together with some more specific matrices. For instance,
the $4\times 4$ unitary matrix $U$ given by
\[
U=\left[ \begin{array}{cccc} 0 & 1& 0 &0  \\ \cos{\theta}& 0& 0 & \sin{\theta} \\
-\sin{\theta} & 0&0 & \cos{\theta}\\ 0 & 0& 1& 0 \end{array} \right], \quad (\theta \in \mathbb R),
\]
is not CMV shaped according to the definition in \cite{KN} but it is CMV-like and it can be
reduced into a pure CMV
shape by a unitary congruence.

In order to investigate the reduction of a unitary matrix $U$ into  CMV--like shape   let us introduce
  the  Hermitian  and the anti-hermitian
part of $U$ defined by $U_H\colon =\displaystyle\frac{U+U^H}{2}$ and
 $U_{AH}=\displaystyle\frac{U-U^H}{2}$ , respectively. If $U$ has at least three distinct eigenvalues then
 the minimal polyanalytic polynomial
$p(z)$ of $U$ is $p(z)=z \bar z-1$, that is, $p(z)$ is the minimal (w.r.t. a fixed order)  degree polynomial in
$z$ and $\bar z$ such that $p(U)=0$.   Observe that
\begin{equation}\label{eqrew}
z \bar z-1= \frac{(z+\bar z)z}{2}-\frac{(z-\bar z)z}{2}-1.
\end{equation}

A block variant of the customary Lanczos method can be used to compute a
block tridiagonal reduction of the Hermitian matrix $U_H$.
The process is defined by the choice of
the initial approximating subspace $\mathcal D_0=<\B z, \B w>$.
A  simple version of the algorithm is given below.

\bigskip
{\footnotesize{
\framebox{\parbox{8.0cm}{
\begin{code1}
{\bf Procedure} {\bf Block\_Lanczos }\\
{\bf Input}: $U_H$, $D_0=[\B z|\B w]$;\\
\ $[G,R,V]={\bf svd}(D_0)$; $s={\bf rank}(R)$; \\
\ $Q(:, 1:s)=G(:,1:s)$; $s0=1, s1=s$; \\
\ {\bf while} $s1<n$\\
\quad  $W=U_H \cdot Q(:, s0:s1)$; $T(s0:s1, s0:s1)= (Q(:, s0:s1))^H\cdot U_H \cdot Q(:, s0:s1)$; \\
\quad   {\bf if} $s0=1$\\
\quad \quad $W=W-Q(:,s0:s1)\cdot T(s0:s1, s0:s1)$;\\
\quad  {\bf else}\\
\quad \quad $W=W-Q(:,s0:s1)\cdot T(s0:s1, s0:s1)$; \\
\quad \quad $W=W-Q(:,\hat s0:\hat s1)\cdot T(\hat s0:\hat s1, s0:s1);$\\
\quad  {\bf end}\\
\quad  $[G,R,V]={\bf svd}(W)$; $snew={\bf rank}(R)$; \\
\quad {\bf if} $snew=0$\\
\quad \quad {\bf disp}('{\rm premature} {\rm stop}'); {\bf return}; \\
\quad {\bf else}\\
\quad  \quad $R=R(1:snew, 1:snew)\cdot (V(:,1:s))^H$; \\
\quad \quad  $\hat s0=s0, \hat s1=s1, s0=s1+1, s1=s1+snew$; \\
\quad \quad $Q(:, s0:s1)=G(:,1:snew)$, $T(s0:s1, \hat s0:\hat s1)=R(1:snew, 1:s)$;\\
\quad \quad $T(\hat s0:\hat s1, s0:s1)=(T(s0:s1, \hat s0:\hat s1))^H$, $s=snew$;\\
\quad {\bf end}\\
\ {\bf end}\\
\ $T(s0:s1,s0:s1)=(Q(:,s0:s1))^H*U_H*Q(:,s0:s1)$;
\end{code1}}}}}
\bigskip

If the procedure {\bf Block\_Lanczos} terminates without premature stop
then at the very end the unitary matrix $Q$ transforms  $U_H$ into the Hermitian block
tridiagonal  matrix $T_H=Q^H \cdot U_H \cdot Q$, where
\[
T_H=\left[\begin{array}{cccc} A_1 & B_1^H\\B_1 & A_2&\ddots \\
& \ddots& \ddots &B_{p-1}^H\\ & & B_{p-1}&A_p\end{array}\right],
\]
with $A_k\in \mathbb C^{i_k\times i_k}$, $B_k\in \mathbb C^{i_{k+1}\times i_k}$,
and $2\geq i_k\geq i_{k+1}$, $i_1+\ldots =i_p=n$. The sequence $\{i_k\}_{k=0}^p$ ($i_0=1$)
is formed from the values attained by the variable $s$ in the procedure {\bf Block\_Lanczos}.

Observe that from $Q\cdot T_H=U_H \cdot  Q$  it follows  inductively that
for the subspaces
\[
\begin{array}{ll}
\mathcal B_j(U_H, Q(:, 1:i_1))=\bs \{Q(:, 1:i_1), U_H\cdot Q(:, 1:i_1), \ldots,U_H^{j-1}\cdot Q(:, 1:i_1)\}\\
\colon =\{\sum_{k=0}^{j-1} U_H^{k} Q(:, 1:i_1)\Psi_k, \ \Psi_k\in \mathbb C^{i_1\times i_1}\}
\end{array}
\]
and
\[
\begin{array}{ll}
\mathcal Q_j=\bs \{Q(:, 1:i_1), Q(:, i_1+1:i_1+i_2), \ldots, Q(:, \sum_{k=1}^{j-1}i_{k}+1: \sum_{k=1}^{j}i_{k})\}\\
\colon =\{\sum_{k=0}^{j-1} Q(:, \sum_{\ell=0}^{k}i_{\ell}+1:
\sum_{\ell=0}^{k+1}i_{\ell})\Phi_k, \ \Phi_k\in \mathbb C^{i_{k+1}\times i_{1}}\}
\end{array}
\]
it holds
\[
\mathcal B_j(U_H, Q(:, 1:i_1)) \subseteq \mathcal Q_j, \quad j\geq 1.
\]
Moreover, since
\begin{equation}\label{plus}
U_H^{j-1}\cdot Q(:, 1:i_1)=
Q(:, \sum_{k=1}^{j-1} i_{k}+1: \sum_{k=1}^{j}i_{k}) B_{j-1} \cdots B_1 + \sum_{\ell= 1}^{j-2}
Q(:, \sum_{k=1}^{\ell-1} i_{k}+1: \sum_{k=1}^{\ell}i_{k}) \Gamma_{\ell},
\end{equation}
and $ B_{k-1} \cdots B_1\in \mathbb C^{i_k\times i_1}$ is of maximal rank $i_k$
since the blocks $B_i$ are obtained with the {{\bf Block\_Lanczos} procedure.  Then
we obtain
\[
\mathcal B_j(U_H, Q(:, 1:i_1))= \mathcal Q_j, \quad j\geq 1.
\]
Conversely,  this relation implies that the matrix $Q^H \cdot U_H \cdot Q$ is block upper
Hessenberg and therefore block-tridiagonal.

Now let us  consider the matrix $H\colon =Q^H \cdot U_{AH} \cdot Q$.
We are going to investigate the structure of this matrix under the additional assumption
that $\B w=U \B z$.  Differently speaking, we suppose that  the initial approximating
subspace $\bs \{Q(:, 1:i_1)\}$  satisfies $\bs \{Q(:, 1:i_1)\}=\bs \{D_0\}$,
with  $D_0=\left[\B z, U\B z\right]$ depending on just one
vector.

In the following lemma we prove that if the {\bf Block\_Lanczos} procedure does not occur in a  breakdown, the size of the blocks $A_k$ and $B_k$ is $2$, with the  exception of the last block in the case $n$ is odd.

\begin{lem}\label{mainc1}
Let us consider the  algorithm {\bf Block\_Lanczos} applied to
$U_H$  with the initial vectors $D_0=[\B z| U \B z]$, $\B z\in \mathbb C^n$ and
 assume that the set of  vectors, for $j\ge 1$,  $U_H^{j-1} \B z,U_H^{j-1}  U \B z $  yields a basis of the whole space
$\mathbb C^n$.  Then the unitary matrix $Q$ satisfies \eqref{teq} with $i_1=\ldots=i_{p-1}=2$
and $i_p\in \{1,2\}$ depending on the parity of $n$.
\end{lem}
\begin{proof}
The proof of the lemma consists in proving that the matrices $W$ generated in the procedure  {\bf Block\_Lanczos} have always rank $2$, except for the last block in the case $n$ is odd.

Assume by contradiction that at step $j$, $j<n/2$ for $n$ even, $j<n/2-1$ for $n$ odd, we have a matrix $W$  such that $\rank(W)<2$, and that in previous steps all the corresponding matrices $W$ had full rank.
The first $2j$ columns of the unitary matrix $Q$ returned by the  procedure {\bf Block\_Lanczos}, form an orthogonal basis for the  space $\mathcal B_j(U_H, Q(:, 1:i_1))$. Since $U$ is unitary it is easy to see that
$\mathcal B_j(U_H, Q(:, 1:i_1))=\mbox{span}\{ \B z, U\B z| U^H\B z, U^2\B z| \ldots,| (U^H)^{j-1}\B z, U^j\B z\}$.
Since $\rank(W)<2$ then
a nonzero linear combination of $(U^H)^{j-1} z$ and $U^j z$ must be in $\mathcal B_{j-1}(U_H, Q(:, 1:i_1))$. Due to the unitariness of $U$, this means that both $(U^H)^j\B z$ and $U^{j+1}\B z$ are in $\mathcal B_j(U_H, Q(:, 1:i_1))$. So the dimension of $\mathcal B_{j+1}(U_H, Q(:, 1:i_1))$ is $2j$ and this implies that $\rank(W)=0$, i.e. a premature termination, which is a contradiction.
\end{proof}

\begin{thm}\label{main}
Let us consider the  algorithm {\bf Block\_Lanczos} applied to
$U_H$  with the initial vectors $D_0=[\B z| U \B z]$, $\B z\in \mathbb C^n$.
Let us assume that the set of  vectors $U_H^{j-1} \B z,U_H^{j-1}  U \B z $, $j\geq 1$, yields a basis of the whole space
$\mathbb C^n$, or, equivalently, the algorithm    does not return an early termination
warning.  Then  the unitary matrix $Q\in \mathbb C^{n\times n}$  reduces simultaneously $U_H$ and
$U_{AH}$ into a block triangular form, that is,
\[
T_H\colon=Q^H U_H Q=\left[\begin{array}{cccc} A_1 & B_1^H\\B_1 & A_2&\ddots \\
& \ddots& \ddots &B_{p-1}^H\\ & & B_{p-1}&A_p\end{array}\right],
\]
and
\[
T_{AH}\colon=Q^H U_{AH}Q =\left[\begin{array}{cccc} A_1' & -B_1'^H\\B_1' & A_2'&\ddots \\
& \ddots& \ddots &-B_{p-1}'^H\\ & & B_{p-1}'&A_p'\end{array}\right]
\]
where  $A_k, A_k'\in \mathbb C^{2\times2}$, $B_k, B_k'\in \mathbb C^{2\times 2}$,
and in the case $n$ is odd, $A_p\in \mathbb C^{1\times1}$, $B_{p-1}, B_{p-1}'\in \mathbb C^{1\times 2}$.
\end{thm}
\begin{proof}
In order to compare  the shape of $H=Q^H\, \cdot \,U_{AH}\,\cdot\, Q$ and $T_H$ we first  introduce a commensurable partitioning of $H$,
i.e.,
$H=(H_{k,\ell}), H_{k,\ell}\in \mathbb C^{2\times2}$, (possibly for odd $n$, $H_{p, \ell}\mathbb C^{1\times2}$, and $H_{p,p}\in \mathbb C$).   Clearly, we  have that
\[
\mathcal B_j(U_{AH}, Q(:, 1:i_1))=\mathcal B_j(U_{AH}, D_0), \quad j\geq 1.
\]
Let us begin by  considering  the initial step $j=1$.
We have
\begin{equation}\label{oneeq}
U_{AH} \B z=-U_H \B z +U\B z\in \bs \{Q(:, 1:i_1), Q(:, i_1+1:i_1+i_2)\}.
\end{equation}
Concerning $U_{AH} U \B z$  from \eqref{eqrew} we obtain that
\begin{equation}\label{twoeq}
U_{AH} U \B z=U_H U \B z -\B z\in \bs \{Q(:, 1:i_1), Q(:, i_1+1:i_1+i_2)\}.
\end{equation}
By using \eqref{oneeq} and \eqref{twoeq} it is found that
\begin{equation}\label{deltaeq}
U_{AH}  Q(:, 1:i_1)= U_H  Q(:, 1:i_1) \Delta +  Q(:, 1:i_1) \Gamma,
\end{equation}
for suitable matrices $\Delta, \Gamma\in \mathbb C^{i_1\times i_1}$ with
$\Delta$ of maximal rank $i_1$.
Since $U_{AH}$ and $U_H$ commute,  from \eqref{plus} and \eqref{deltaeq} by induction it follows that
\[
\mathcal B_j(U_{AH}, Q(:, 1:i_1)) \subseteq \mathcal Q_j, \quad j\geq 1,
\]
and
\begin{equation}\label{uahshape}
U_{AH}^{j-1}\cdot Q(:, 1:i_1)=
Q(:, \sum_{k=1}^{j-1} i_{k}+1: \sum_{k=1}^{j}i_{k}) B_{j-1} \cdots B_1 \Delta^{j-1}+ \sum_{\ell= 1}^{j-2}
Q(:, \sum_{k=1}^{\ell-1} i_{k}+1: \sum_{k=1}^{\ell}i_{k}) \Gamma'_{\ell},
\end{equation}
for  suitable matrices $\Gamma_{\ell}$, which says that
\[
\mathcal B_j(U_{AH}, Q(:, 1:i_1)) =\mathcal Q_j, \quad j\geq 1,
\]
and, therefore, the matrix $H$ is block  upper Hessenberg, i.e., $H_{k, \ell}=0$ for $k-1>\ell$.
On the other hand,  since
$U_{AH}$ is anti-hermitian we find that also $H$ is anti-hermitian and, hence,
$H$ is block  tridiagonal with the same shape as $T_H$. Finally, this implies that both
$Q^H\cdot U\cdot Q$ and $Q^H \cdot U^H\cdot Q$ are in block tridiagonal form with the
same shape as $H$ and $T_H$.
\end{proof}

From this result it follows that the unitary transformation induced by the matrix $Q$
has the same action   when performed on the unitary matrix $U=
\displaystyle\frac{U_H +U_{AH}}{2}\in \mathbb C^{n\times n}$. i.e,
\begin{equation}\label{teq}
T\colon=Q^H U Q= \left[\begin{array}{cccc} \hat A_1 & \tilde B_1^H\\\hat B_1 & \hat A_2&\ddots \\
& \ddots& \ddots &\tilde B_{p-1}^H\\ & & \hat B_{p-1}&\hat A_p\end{array}\right]
\end{equation}
where  $\hat A_k\in \mathbb C^{2\times 2}$, $\hat B_k=\frac{B_k + B_k'}{2}\in \mathbb C^{2\times 2},
\tilde B_k^H=\frac{B_k^H -B_k'^H}{2}\in \mathbb C^{2\times 2}$, for $k=1, \ldots, \lfloor n/2\rfloor$, 
and in the case $n$ is odd, $\hat A_p\in \mathbb C, \hat B_{p-1}, \tilde B_{p-1}\in \mathbb C^{1\times 2}$.
As proved in Lemma~\ref{mainc1}, blocks $B_k$ of $T_H$ are $2\times 2$ and have full rank.

In next step  of our investigation on the shape  of the matrix $T$ we give
a careful look at the rank of its out-of-diagonal blocks.

\begin{corollary}\label{mainc2}
Let us consider the  algorithm {\bf Block\_Lanczos} applied to
$U_H$  with the initial vectors $D_0=[\B z| U \B z]$, $\B z\in \mathbb C^n$ and
 assume that the set of vectors $U_H^{j-1} \B z, U_H^{j-1} U \B z$, $j\geq 1$, yields a basis of the whole space
$\mathbb C^n$.  Then both the subdiagonal blocks $\hat B_k$ and the superdiagonal blocks
$\tilde B_k^H$, $1\leq k\leq p-1$,  of the matrix $T=Q^H U Q$ are $2\times 2$ rank-one matrices.
\end{corollary}
\begin{proof}
We consider only the case $n$  even, the case $n$ odd is similar.
From \eqref{oneeq} and \eqref{twoeq} it follows that
the matrix $\Delta\in \mathbb C^{2\times 2}$ in \eqref{deltaeq} can be represented as
\[
\Delta=\Theta \left[ \begin{array}{cc} -1 & \\& 1\end{array}\right] \Theta^{-1},
\]
for a suitable invertible matrix $\Theta  \in \mathbb C^{2\times 2}$. This means that $I_2+ \Delta$ and $I_2-\Delta$ are
rank-one matrices.   From
\[
U_{AH} U_H^{j-1} Q(\colon, 1:i_1)=U_H^{j-1}U_{AH} Q(\colon, 1:i_1), \quad j\leq 1,
\]
by using \eqref{plus}, \eqref{deltaeq}, \eqref{uahshape}
it is found that the matrices $B_j'$ , $1\leq j\leq p-1$, satisfy the following relations
\[
B_j'=B_j \cdot B_{j-1}\cdots B_1 \cdot \Delta \cdot B_1^{-1} \cdots B_{j-1}^{-1}, \quad 1\leq j\leq p-1.
\]
 There follows that for the  out-of-diagonal  blocks of $T$ it holds
\[
\hat B_j=\frac{B_j + B_j'}{2}=\frac{1}{2}(B_j \cdot  B_{j-1}\cdots B_1 \cdot (I_2 + \Delta) \cdot B_1^{-1} \cdots B_{j-1}^{-1}),
\quad 1\leq j\leq p-1,
\]
\[
\tilde B_j=\frac{B_j - B_j'}{2}=\frac{1}{2}(B_j \cdot  B_{j-1}\cdots B_1 \cdot (I_2 - \Delta) \cdot B_1^{-1} \cdots B_{j-1}^{-1}),
\quad 1\leq j\leq p-1,
\]
which says that $\hat B_j$  and $\tilde B_j$ are  matrices of rank one.
\end{proof}

The  Figure \ref{f1} illustrates the shape of the matrix $T$ generated by the algorithm
{\bf Block\_Lanczos} applied to the generator  of order 16
of the circulant  matrix algebra, that is, the companion matrix associated with the polynomial $z^{16}-1$.
The initial vector $\B z$ is randomly generated and the matrix $T$ is post-processed by means of a sequence of
Givens rotation matrices defined  in order to compress the rank-one structure of the out-of-diagonal blocks.

\begin{figure}
\begin{center}
\includegraphics[scale=0.60]{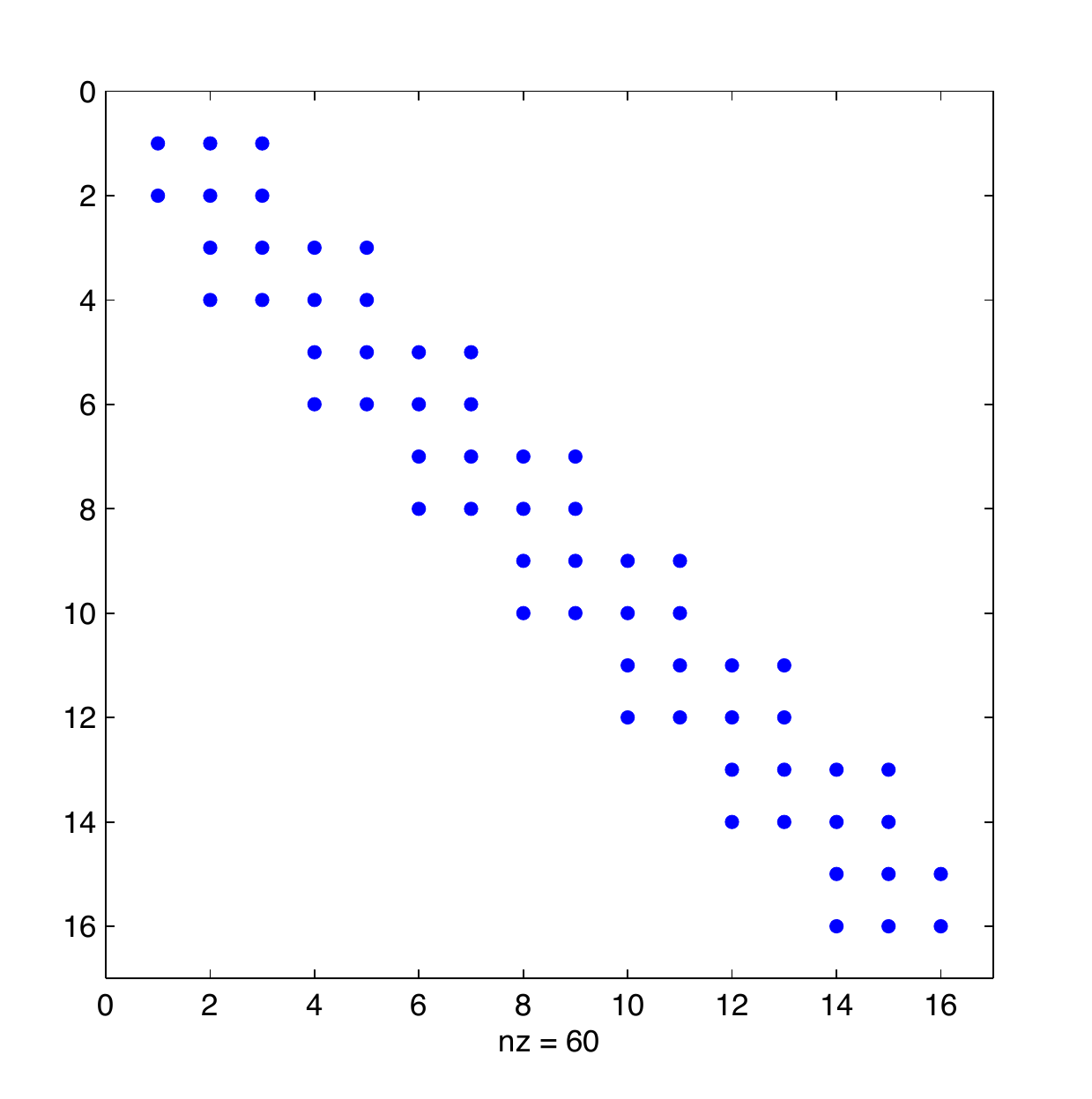}
\caption{Shape of the modified matrix $T$ generated by {\bf Block\_Lanczos}}
\label{f1}
\end{center}
\end{figure}

It is easily seen that under some genericity condition  the  modified matrix $T$
can be reduced  by a diagonal unitary congruence  into the CMV-shape
described in Definition 1.2 in \cite{KN}. This explains why  we refer to the  matrix $T$ generated
by the  {\bf Block\_Lanczos} procedure   as to a CMV--like form. Further, it is worth noticing that
the computational cost of the {\bf Block\_Lanczos}  algorithm  amounts  to  perform
a matrix-by-vector multiplication per step.
Therefore, the overall cost is $O(n c_{m})$, where $O(c_{m})$ is the cost of multiplying the matrix
$U+U^H\in \mathbb C^{n\times n}$  by a vector.
In the case of the generator of the circulant matrix algebra, since the matrix $U$  is sparse,  we get the complexity estimate  $O(n^2)$.

Some comments are however  in order here.  The first 
 one is concerned with the applicability of the {\bf Block\_Lanczos} process devised above.
Let us consider the Fourier matrix $U=\mathcal F_n=\frac{1}{\sqrt{n}}\Omega_n$ of order $n=2m$.
Due to  the relation $\Omega_n^2=n \Pi$, where $\Pi$ is a suitable symmetric permutation matrix, it is
found that for any starting vector $\B z$ the {\bf Block\_Lanczos} procedure breaks down within the first three
steps.  In this case  the reduction scheme has to be restarted and
the initial matrix can be converted into the direct sum of CMV-like blocks.

The second comment  is that the profile shown in Figure \ref{f1} does not
follow immediately from the rank-one property of the out-of-diagonal blocks.  The observation is that
after having performed the first two elementary transformations $T\rightarrow T_1=\mathcal G_1 T\mathcal G_1^H$
and $T_1\rightarrow T_2=\mathcal G_2 T_1\mathcal G_2^H$ determined so that $\hat B_1$ and $\tilde B_1^H$ are converted
in the form displayed in Figure \ref{f1} then the subdiagonal block in  position $(3,2)$ has already the desired
shape due to the unitariness of the matrix $T_2$.  In this way the reduction of $T$ can be completed by a sequence of
Givens rotations only acting on the superdiagonal blocks in such a way to preserve the zero structure.

The final observation is that the reduction into  the CMV--like shape  has also a  remarkable consequence in 
terms of rank structures of the  resulting matrix.   The property follows from a  corollary  (Corollary 3 in \cite{FM}) 
of a well known result in linear algebra referred to as 
the {\em Nullity Theorem}.

\begin{thm}[Nullity Theorem]\label{nullycons}
Suppose $A\in \mathbb C^{n\times n}$ is a nonsingular matrix and $\B \alpha$ and $\B \beta$  to be
nonempty proper subsets of $J\colon =\{1,\ldots, n\}$.  Then
\[
\rank(A^{-1}(\B \alpha; \B \beta))=\rank(A(J\setminus \B \beta; J\setminus \B \alpha)) + |\B \alpha| +|\B \beta| -n,
\]
where, as usual, $|J|$ denotes the cardinality of the set $J$.
\end{thm}
 
Now let $T\in \mathbb C^{n\times n}$, $n=2\ell$, 
  denote a unitary matrix given in the CMV-like form shown in  Figure \ref{f1}. 
 Set $\B \alpha=\{1,2\}$ and $\B \beta=\{5,\ldots, n\}$.
Thus, we have
\[
0=\rank(T(1\colon 2; 5\colon n))=\rank(T^H(5\colon n; 1\colon 2))=
\rank(T(3\colon n; 1\colon 4)) -2, 
\]
which implies 
\[
\rank(T(3\colon n; 1\colon 4))=2.
\]
and, therefore, 
\[
\rank(T(3\colon 4; 2\colon 3))=1
\]
whenever $T(5:6, 4)$ is nonzero. In fact, assume by contradiction that $\rank(T(3\colon 4; 2\colon 3))=2$, then since by Corollary~\ref{mainc2} $\rank(T(3\colon 4; 1\colon 2))=1$, $\rank(T(5\colon:6, 4:5))=1$ and the matrix is in CMV-like form, we obtain that $\rank(T(3:n, 1:4))= 3$ which is a contradiction.
 By similar reasonings we conclude that 
under the assumption that $T(2k+1:2(k+1), 2k)$ is nonzero for $k=2,\ldots,\ell-1$ it is obtained that 
\begin{equation}\label{rr1}
\rank(T(2k+1\colon 2(k+1); 2k\colon 2k+1))=1, \quad 1\leq k\leq \ell-2.
\end{equation}
This property plays a fundamental role in the analysis of QR--based eigensolvers performed in Section \ref{invv}.

\section{A Householder style reduction algorithm}
In this section we derive a Householder style reduction algorithm using elementary
matrices to perform the
transformation of a unitary matrix $U$ into the block tridiagonal form $T$.
 Let us recall that from \eqref{teq} the unitary matrix $U$ is  converted into the the block
tridiagonal form $T$ by using a unitary congruence defined by the unitary matrix $Q$  whose first two
columns are determined to generate the space $<\B z, U\B z>$ for a suitable $\B z\in \mathbb C^{n\times n}$.
 The rank properties of the out-of-diagonal blocks of $T$ stated in Corollary \ref{mainc2}  enables the  information
of $T$ to be further compressed by using a sequence of Givens rotations in order to arrive at the form
shown in Figure \ref{f1}.  The pattern of this sequence does not modify the first two column of $Q$.

A straightforward extension  of the Implicit Q Theorem (see \cite{MC}, Theorem 7.4.2) for block Hessenberg
 matrices indicates that essentially the same unitary matrix should be determined in any block tridiagonal
reduction process by means of  unitary transformations applied to $U+U^H$ provided that the first two columns of $Q$ span the space $<\B z, U\B z>$.  This suggests the  alternative  more stable procedure using
unitary transformations outlined below. For the sake of simplicity we consider the case where $n$ is even.

\bigskip

{\footnotesize{
\framebox{\parbox{8.0cm}{
\begin{code1}
 {\bf Procedure} {\bf Unitary\_CMV\_Reduction }\\
{\bf Input}: $U$, $D_0=[\B z|U \B z ]$;\\
\ $[G,R]={\bf qr}(D_0)$; $U=G^H U G$; $n={\bf length}(U)$;\\
\ {\bf for } $kk=1:n/2-2$\\
\quad {\bf for} $j=n:-1:2 kk+3$\\
\quad \quad     $Us=U(j-2:j, 2 kk-1:2 kk) + (U(2 kk-1:2 kk, j-2:j))^H$;\\
\quad \quad     $[Qs,Rs]={\bf qr}(Us)$; \\
 \quad \quad    $U(j-2:j, :)=Qs^H U(j-2:j,:)$;\\
 \quad \quad    $U(:, j-2:j)=U(:,j-2:j)Q_s$;\\
  \quad   {\bf end}\\
\ {\bf end}
\end{code1}}}}}

\bigskip
The complexity of this procedure is generally $O(n^3)$. However there are some interesting structures which can be
exploited. In particular, if $U$ is rank-structured then $U+U^H$ is also rank-structured and
Hermitian so that  the fast techniques  developed in  \cite{EGG} might
 be incorporated  in the algorithm above by achieving a
quadratic complexity. The resulting fast  version of   the {\bf  Unitary\_CMV\_Reduction}
algorithm will be presented elsewhere.

It is also worth noticing that  for certain input data the
algorithm above  needs to be restarted if breakdown situations are
detected.  For instance,
in Figure \ref{f2} we report the matrix  generated by the algorithm {\bf  Unitary\_CMV\_Reduction} applied to the Fourier
matrix of order 32.  The $8\times 8$ block diagonal structure of the CMV--shape corresponds with
7 restarts of the algorithm.  The 2-norm of the subdiagonal blocks corresponding to breakdowns are in the range
$[2.3e-14, 6.5e-14]$.  The deflation criterion used in our implementation  compares the 2-norm of the  considered
subdiagonal block with  the  bound $n \cdot {\rm u} \cdot \parallel U\parallel_{F}$, where  u denotes the unit roundoff.
This is justified by the fact that the Hessenberg reduction as well as  QR iterations introduce roundoff errors of
the same level \cite{Tisse}.

\begin{figure}
\begin{center}
\includegraphics[scale=0.60]{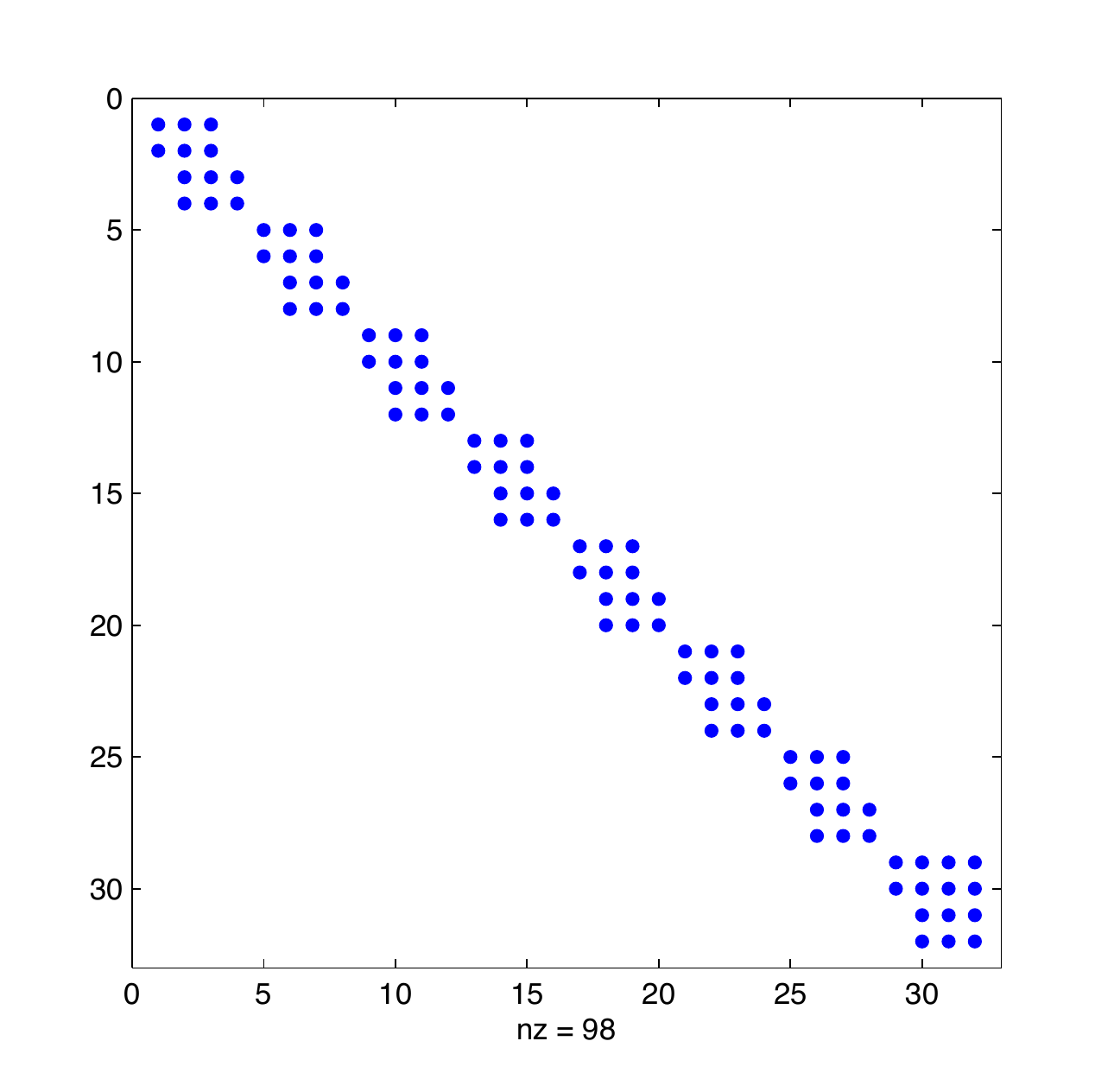}
\caption{Block CMV reduction for the Fourier matrix of size $n=32$}
\label{f2}
\end{center}
\end{figure}

\section{Invariance of the CMV-like form under the QR algorithm}\label{invv}

Our  interest towards the reduction of unitary matrices into a CMV-like form  stems from the unitary
eigenvalue  problem. In this respect it is mandatory to investigate the properties of CMV-like matrices under
the QR  algorithm which  yields  the  customary approach   for eigenvalue computation.
The invariance of the CMV-like form under the QR iteration was firstly proved in the paper \cite{BE91} using an
algorithmic approach.  The proof presented here is based on the property \eqref{rr1}
and it has the advantage of admitting an immediate extension to almost unitary matrices.

\begin{thm}\label{maininv}
Let $T_0=T\in \mathbb C^{n\times n}$, $n=2\ell$, 
 be a unitary  block tridiagonal matrix given  in  the CMV--like form shown in Figure \ref{f1}.
Moreover, let  $T_1$  denote the matrix generated from one step of the
QR algorithm with shift $\gamma$ , i.e.,
\[
\left\{ \begin{array}{ll}
T_0-\gamma I_n= Q R \quad ({\rm QR} \ \  {\rm  factorization}); \\
T_1=R Q + \gamma I_n.\end{array}\right.
\]
Assume that the subdiagonal blocks of both $T_0$ and $T_1$ are nonzero and 
that $T_0-\gamma I_n$ is nonsingular.  Then the matrix $T_1$ is block
tridiagonal with the same shape as $T_0$.
\end{thm}
\begin{proof}
From $T_1=R T_0 R^{-1}$ where $R$ is upper triangular it follows that both the
 shape  and the rank properties of  the
lower triangular  portion  of the matrix $T_0$ are preserved in the matrix $T_1$.  Then 
by using Theorem \ref{nullycons} and property \eqref{rr1}  we obtain that 
\[
\rank(T_1(1\colon 2; 4\colon n))=\rank(T_1^H(4\colon n; 1\colon 2))=
\rank(T_1(3\colon n; 1\colon 3)) -1.
\]
We have that $\rank(T_1(3\colon n; 1\colon 3)) =\rank(T_0(3\colon n; 1\colon 3)) =1$ as a consequence of the CMV-like structure of $T_0$ and for \eqref{rr1}.
This says that $T_1(1\colon 2; 4\colon n)$ is a zero matrix.  The same argument applies to any
submatrix $T_1(1\colon 2s; 2(s+1) \colon n)$, $1\leq s\leq \ell-2$,   
by yielding the CMV-like structure of $T_1$.

\end{proof}

The assumption about the invertibility of  $ T_0-\gamma I_n$ is just for the sake of simplicity. In the
singular case  we can get the same conclusion  by a careful looking at the  elementary transformations involved
in the passage from $T_0$ to $T_1$ (compare with \cite{DV_sing} for a general treatment of rank structures under the
QR process in the singular case). 

The conclusion of Theorem \eqref{maininv} immediately generalizes  to almost unitary matrices of the form
  $A=U +\B z\B w^H$, with $U$ unitary, including the class of companion matrices.
 Indeed,  if we apply  the {\bf  Unitary\_CMV\_Reduction} to the matrix $U$ with
initial vectors $\B z, U\B z$, then it is shown that $U$ is reduced to CMV-like shape and simultaneously
the rank-one correction is  converted to a  rank-one matrix with only the first nonzero row.  In this way, the
transformed matrix $B=Q^H A Q=T + \B e_1 \B v^H$ is block upper Hessenberg with subdiagonal blocks of rank one.
Setting $T_0=T$, $C_0=\B e_1 \B v^H$, and applying a $QR$ step to matrix $B_0=B$, we have
$$
B_1=T_1+C_1,
$$
where $T_1=Q^HT_0Q$ is unitary and  $C_1=Q^HC_0Q$ is a rank one matrix.
The following facts hold. 

\begin{itemize}
\item Since the lower shape is maintained under the QR  scheme, $B_1$ is  block upper Hessenberg with subdiagonal blocks of rank one.
\item $T_1$ has a rank--one structure in the lower triangular portion out of the diagonal blocks, due to the fact that $B_1$ is block Hessenberg with rank one subdiagonal blocks.
\item $\rank(T_1(3:n,1:4))=3$, and for the Nullity Theorem, $\rank(T_1(1:2, 5:n))=1$.
\item Then, $\rank(B_1(1:2, 5:n))\le 2$.
\end{itemize}
This reasoning can be repeated at each $QR$ step obtaining that the matrix $B$ generated at each  step  has a rank--two
structure in the upper triangular portion out of the block tridiagonal profile.



As an illustration of the structural properties  of the CMV-like reduction
under the QR process, we  consider in
Figure \ref{f4}  the unitary matrix $T$
computed  after 32 QR  steps  applied  to  the matrix $B$ determined from a companion matrix $A$
generated by the Matlab command ${\it compan}$. Specifically, we show the shape of the matrices  $H$ and $S$
such that $T\cdot S=H$ and $S$ can be represented as product of a linear number of $2\times 2$ Givens rotation matrices.
This factorization is an intermediate step in the computation of a QR factorization of $T$  \cite{EG02} and
clearly put in evidence the rank-one structure  in the upper triangular portion of the matrix $T$.

\begin{figure}
 \begin{minipage}[b]{6.0cm}
   \centering
   \includegraphics[width=7cm]{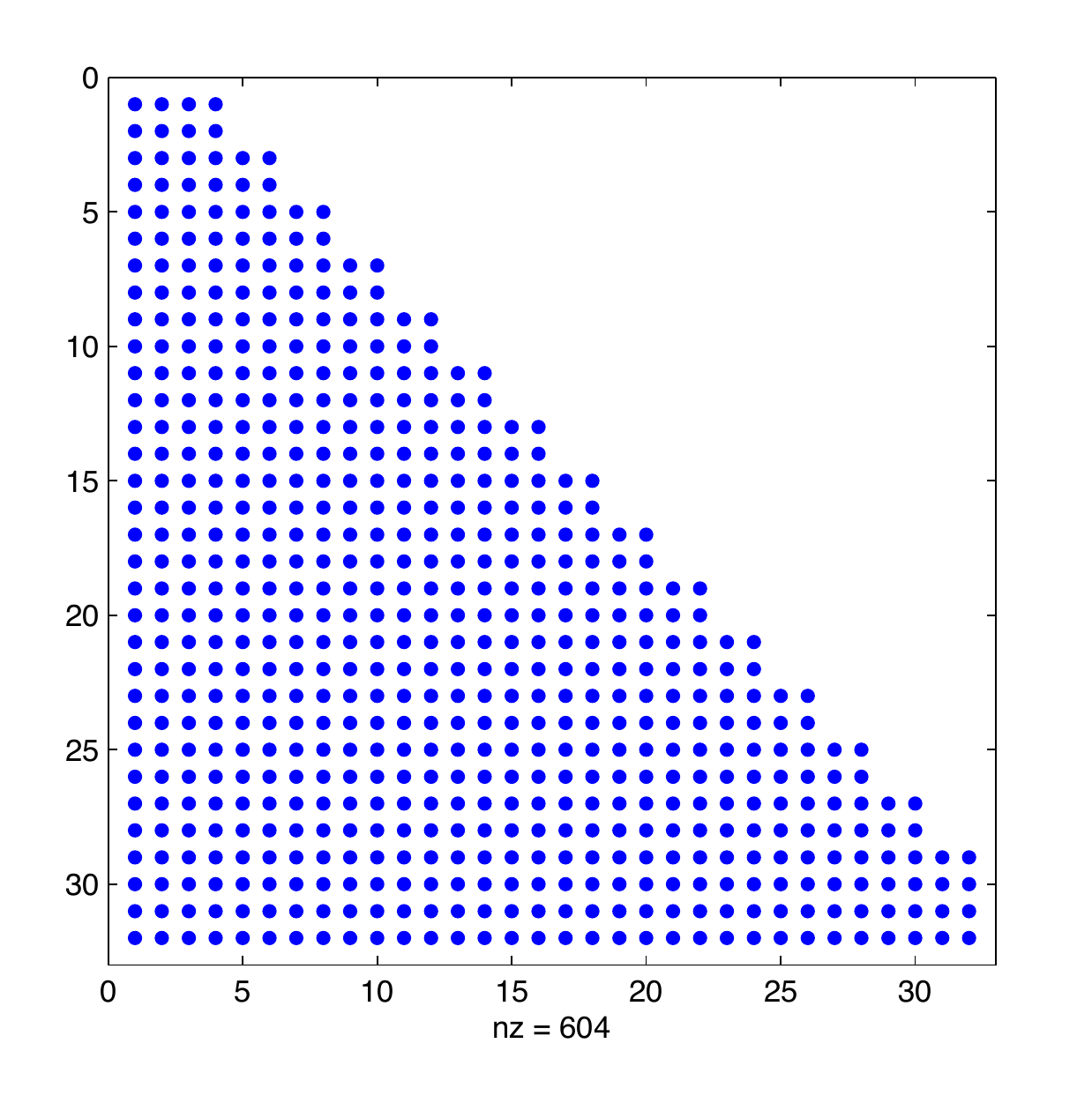}
   \caption{Shape of the matrix $H$}
\label{ff1}
 \end{minipage}
 \hspace{20mm}
 \begin{minipage}[b]{6cm}
  \centering
   \includegraphics[width=7.25cm]{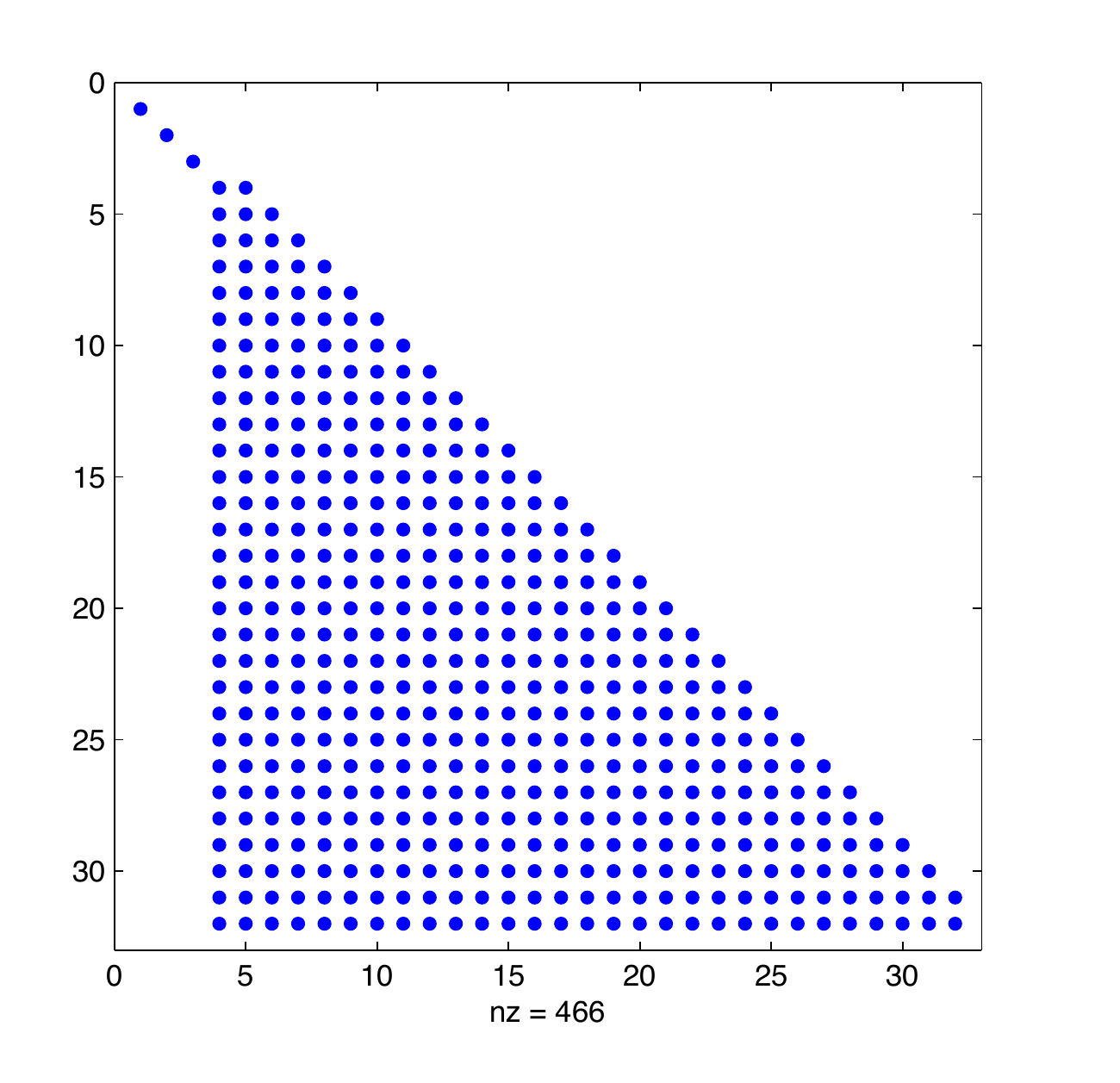}\vspace{2.5mm}
   \caption{Shape of the matrix $S$}
 \end{minipage}
\caption{Illustration of the upper rank-one structure of $T$}
\label{f4}
\end{figure}

\section{Conclusion and future work}
In this paper we have presented fast numerical methods for reducing a given
 unitary matrix into the direct sum of CMV-like matrices.   The occurrence of
potential rank  structures  in the matrix $U$ can be exploited  by improving the computational cost of the reduction.
As the  resulting sparse representation of $U$
 is invariant under the QR method  the proposed reduction can naturally  be applied for the
fast solution of the unitary rank structured eigenvalue problems.

Some theoretical and computational issues are currently being investigated. A
first interesting question is related
with the completion problem for unitary matrices  whose block lower Hessenberg
profile  is only known (compare with \cite{BEGG_comp}). 
Other issues  concern the numerical treatment,
representation and compression of  out-of-band semiseparable   unitary matrices \cite{EG_band}. In particular, 
the rank--one structure  of the matrices generated under the QR method applied to a transformed companion matrix 
can provide an easy, numerically efficient representation  to be exploited for the design of 
 fast matrix--based polynomial rootfinding methods. The preliminary transformation of  a companion matrix 
 into a
perturbed CMV-like form  can
provide alternative algorithms with better numerical and computational features.


\end{document}